\documentclass[12pt]{amsart}
\usepackage{amsthm}
\usepackage{amsmath}
\usepackage{amssymb}

\usepackage{amsfonts}
\usepackage{graphicx}
\usepackage[all, knot]{xy}

\usepackage{tikz}

\usepackage{hyperref}

\usepackage{caption}
\usepackage{subcaption}

\usepackage[margin=1.2in]{geometry}

\theoremstyle{plain}

\newtheorem{theorem}{Theorem}[section]
\newtheorem{lemma}[theorem]{Lemma}
\newtheorem{proposition}[theorem]{Proposition}

\newtheorem*{theorem*}{Theorem}
\newtheorem*{lemma*}{Lemma}
\newtheorem*{proposition*}{Proposition}
\newtheorem*{corollary*}{Corollary}

\theoremstyle{definition}

\newtheorem{rmk}[theorem]{Remark}

\newtheorem*{definition*}{Definition}
\newtheorem*{example*}{Example}
\newtheorem*{construction*}{Construction}
\newtheorem*{notation*}{Notation}
\newtheorem*{rmk*}{Remark}

\newcommand{\co}{\colon}

\begin{document}

\title{Hexatonic Systems and Dual Groups in Mathematical Music Theory}

\author{Cameron Berry}
\address{Department of Mathematics, Michigan State University,
619 Red Cedar Road, East Lansing, MI 48824}
\email{berrycam@msu.edu}

\author{Thomas M. Fiore}
\address{Thomas M. Fiore, Department of Mathematics and Statistics,
University of Michigan-Dearborn, 4901 Evergreen Road, Dearborn,
MI 48128 \newline  and \newline
NWF I - Mathematik,
Universit\"at Regensburg,
Universit\"atsstra{\ss}e 31,
93040 Regensburg,
Germany }
\email{tmfiore@umich.edu}
\urladdr{http://www-personal.umd.umich.edu/~tmfiore/}

\begin{abstract}
Motivated by the music-theoretical work of Richard Cohn and David Clampitt on late-nineteenth century harmony, we mathematically prove that the $PL$-group of a hexatonic cycle is dual (in the sense of Lewin) to its $T/I$-stabilizer.
Our point of departure is Cohn's notions of maximal smoothness and hexatonic cycle, and the symmetry group of the 12-gon; we do {\it not} make use of the duality between the $T/I$-group and $PLR$-group. We also discuss how some ideas in the present paper could be used in the proof of $T/I$-$PLR$ duality by Crans--Fiore--Satyendra in \cite{cransfioresatyendra}.
\end{abstract}

\thanks{A slightly revised version of this manuscript will be published in {\it Involve: A Journal of Mathematics}, published by Mathematical Sciences Publishers.}

\maketitle


\section{Introduction: Hexatonic Cycles and Associated Dual Groups} \label{sec:Introduction}
Why did late nineteenth century composers, such as Franck, Liszt, Mahler, and Wagner, continue to privilege consonant triads over other tone collections, while simultaneously moving away from the diatonic scale and classical tonality?

Richard Cohn proposes an answer in \cite{cohn1996}, independent of acoustic consonance: major and minor triads are preferred because they can form {\it maximally smooth cycles}. Consider for instance the following sequence of consonant triads, called a {\it hexatonic cycle} by Cohn.
\begin{equation} \label{equ:hexatonic_cycle_Eflat}
E\flat,e\flat,B,b,G,g,E\flat
\end{equation}
We have indicated major chords with capital letters and minor chords with lowercase letters.
Although the motion from a major chord to its parallel minor, e.g. $E\flat$ to $e\flat$, $B$ to $b$, and $G$ to $g$, is distinctly non-diatonic, this sequence has cogent properties of importance to late-Romantic composers, as axiomatized
in Cohn's notion of maximally smooth cycle \cite[page 15]{cohn1996}.
\begin{itemize}
\item
It is a {\it cycle} in the sense that the first and last chords are the same but all others are different. A {\it cycle} is required to contain more than three chords.
\item
All of the chords are in one ``set class''; in this case each chord is a consonant triad.
\item
Every transition is {\it maximally smooth} in the sense that two notes stay the same while the third moves by the smallest possible interval: a semitone.
\end{itemize}

Cohn considered movement along this sequence transformationally as an action by a cyclic group of order 6. Additionally, David Clampitt considered in \cite{clampittParsifal} movement along this sequence via $P$ and $L$, and also via certain rotations and reflections. As usual, we denote by $P$ the ``parallel'' transformation that sends a major or minor chord to its parallel minor or major chord, respectively. We denote by $L$ the ``leading tone exchange'' transformation, which moves the root of a major chord down a semitone and the fifth of a minor chord up a semitone, so the $L$ sends consonant triads $e\flat$ to $B$, and $b$ to $G$, and $g$ to $E\flat$. The hexatonic cycle \eqref{equ:hexatonic_cycle_Eflat} is then positioned in network \eqref{equ:PL_network_of_hexatonic_cycle}, with alternating $P$ and $L$ transformations between the nodes.
\begin{equation} \label{equ:PL_network_of_hexatonic_cycle}
\entrymodifiers={=<1.7pc>[o][F-]}
\xymatrix@C=3.5pc@R=3.5pc{E\flat \ar[r]^{P}  &  e\flat \ar[r]^L & B \ar[r]^P & b \ar[r]^L & G \ar[r]^P & g \ar@/^2pc/[lllll]^L }
\end{equation}

Wagner's Grail motive in {\it Parsifal} can be interpreted in terms of the network \eqref{equ:PL_network_of_hexatonic_cycle}, as proposed by David Clampitt in \cite{clampittParsifal}. A small part of Clampitt's analysis of the first four chords is pictured in Figure~\ref{fig:grail_notes_from_Clampitt_and_its_PL_interpretation}. Clampitt includes the final $D\flat$ chord, which lies outside of the hexatonic cycle \eqref{equ:hexatonic_cycle_Eflat}, into his interpretation via a conjugation-modulation applied to a certain subsystem. A third interpretation, in addition to the cyclic one of Cohn \cite[Example~5]{cohn1996} and the $PL$ interpretation in Figure~\ref{fig:grail_notes_from_Clampitt_and_its_PL_interpretation}, was also proposed by Clampitt, this time in terms of the transpositions and inversions $\{T_0,T_4,T_8,I_1,I_5,I_9\}$. Clampitt observes that this group and the $PL$-group form {\it dual groups in the sense of Lewin} \cite{LewinGMIT}, via their actions on the hexatonic set of chords in \eqref{equ:hexatonic_cycle_Eflat}. The perceptual basis of all three groups is explained in \cite{clampittParsifal}.

\begin{figure}[h]
    \centering
        \begin{subfigure}[b]{0.4\textwidth}
        \centering
        $E\flat$, $b$, $G$, $e\flat$, $D\flat$ \vspace{.3in}
        \caption{Chord sequence in Grail motive from Wagner, {\it Parsifal}, Act 3, measures 1098--1100, see Clampitt \cite[Example 1]{clampittParsifal}.}
    \end{subfigure}%
   \hspace{.2in}
    \begin{subfigure}[b]{0.55\textwidth}
        \centering
        \begin{math}
\entrymodifiers={=<1.7pc>[o][F-]}
\xymatrix@C=3pc@R=3.5pc{E\flat \ar[r]^{PLP} \ar@/_2pc/[rrr]_{P}  &  b \ar[r]^L & G \ar[r]^{PLP} & e\flat}
\end{math}
        \caption{First four chords of the Grail motive in a hexatonic $PL$-network of Clampitt. Notice that the bottom arrow is the composite of the three top arrows, and goes in the opposite direction of the bottom arrow of diagram \eqref{equ:PL_network_of_hexatonic_cycle}.}
    \end{subfigure}
    \caption{Wagner's Grail motive and Clampitt's interpretation of first four chords using the $PL$-network in diagram \eqref{equ:PL_network_of_hexatonic_cycle}.} \label{fig:grail_notes_from_Clampitt_and_its_PL_interpretation}
\end{figure}

The contribution of the present article is to directly prove that the $PL$-group and the group $\{T_0,T_4,T_8,I_1,I_5,I_9\}$ in Clampitt's article are dual groups acting on \eqref{equ:hexatonic_cycle_Eflat}. Our point of departure is the hexatonic cycle \eqref{equ:hexatonic_cycle_Eflat}, the standard action of the dihedral group of order 24 on the 12-gon, and the Orbit-Stabilizer Theorem. We do {\it not} use the duality of the $T/I$-group and $PLR$-group.  Some arguments in Section~\ref{sec:Main_Theorem} are similar to arguments of Crans--Fiore--Satyendra \cite{cransfioresatyendra}, but there are important differences, see Remark~\ref{rem:differences_and_similarities}.

Just how special are the consonant triads with regard to the maximal smoothness property? According to \cite{cohn1996}, only 6 categories of tone collections support maximally smooth cycles: singletons, consonant triads, pentatonic sets, diatonic sets, complements of consonant triads, and 11-note sets. Clearly the singletons and 11-note sets do not give musically significant cycles. The pentatonic sets and the diatonic sets each support only one long cycle, which exhausts all 12 of their respective exemplars. The consonant triads and their complements, {\it on the other hand,} support short cycles that do not exhaust all of their transpositions and inversions. The maximally smooth cycles of consonant triads are enumerated as sets as follows.
\begin{equation} \label{equ:Hex1}
\{E\flat,e\flat,B,b,G,g\}
\end{equation}

\begin{equation} \label{equ:Hex2}
\{E,e,C,c,A\flat,a\flat\}
\end{equation}

\begin{equation} \label{equ:Hex3}
\{F,f,C\sharp, c\sharp, A,a\}
\end{equation}

\begin{equation} \label{equ:Hex4}
\{F\sharp,f\sharp,D,d,B\flat,b\flat\}
\end{equation}

These are the four {\it hexatonic cycles} of Cohn, see \cite[page 17]{cohn1996}. They (and their reverses) are the only short maximally smooth cycles that exist in the Western chromatic scale.

\section{Mathematical and Musical Preliminaries: Standard Dihedral Group Action on Consonant Triads and the Orbit-Stabilizer Theorem} \label{sec:Preliminaries}

We quickly recall the standard preliminaries about consonant triads, transposition, inversion, $P$, $L$, and the Orbit-Stabilizer Theorem. A good introduction to this very well-known background material is \cite{cransfioresatyendra}. Since this background has been treated in many places, we merely rapidly introduce the notation and indicate a few sources.

\subsection{Consonant Triads}
We encode pitch classes using the standard $\mathbb{Z}_{12}$ model, where $C=0$, $C\sharp=D\flat=1$, and so on up to $B=11$. Via this bijection we freely refer to elements of $\mathbb{Z}_{12}$ as {\it pitch classes}. Major chords are indicated as ordered 3-tuples in $\mathbb{Z}_{12}$ of the form $\langle x,x+4,x+7 \rangle$, where $x$ ranges through $\mathbb{Z}_{12}$. Minor chords are indicated as 3-tuples $\langle x+7,x+3,x\rangle$ with $x \in \mathbb{Z}_{12}$. We choose these orderings to make simple formulas for $P$ and $L$, this is not a restriction for applications, as the framework was extended in \cite{FioreNollSatyendraPermutations} to allow any orderings. We call the set of 24 major and minor triads $\text{\it Triads}$, this is the set of {\it consonant triads}. The letter names are indicated in Table~\ref{table:consonant_triads}.
\begin{table}
  \centering
  $\begin{tabular}{|r|l|}
\hline \text{Major Triads} & \text{Minor Triads} \\ \hline
$C=\langle 0,4, 7\rangle$ & $\langle 0,8,5 \rangle = f $ \\
$C \sharp = D \flat=\langle 1,5, 8\rangle$ & $\langle 1,9, 6 \rangle=f \sharp=g \flat$ \\
$D=\langle 2,6, 9\rangle$ & $\langle 2,10,7\rangle = g$ \\
$D \sharp=E \flat=\langle 3,7, 10\rangle$ & $\langle3, 11,8 \rangle=g \sharp=a \flat$ \\
$E=\langle 4,8, 11\rangle$ & $\langle 4,0, 9\rangle=a$ \\
$F=\langle 5,9,  0 \rangle$ & $\langle 5,1,10  \rangle=a \sharp=b \flat$ \\
$F \sharp=G \flat=\langle 6,10, 1 \rangle$ & $\langle 6,2,11\rangle=b$ \\
$G=\langle 7,11, 2 \rangle$ & $\langle 7,3,0\rangle =c$ \\
$G \sharp=A \flat=\langle 8,0, 3 \rangle$ & $\langle8, 4,1\rangle =c \sharp=d \flat$ \\
$A=\langle 9,1, 4 \rangle$ & $\langle9, 5,2\rangle =d$ \\
$A \sharp=B \flat=\langle 10,2,5 \rangle$ & $\langle 10,6,3\rangle=d \sharp=e \flat$ \\
$B =\langle 11,3,6 \rangle$ & $\langle 11, 7,4\rangle =e$ \\
\hline
\end{tabular}$
  \caption{The set of consonant triads, denoted $\text{\it Triads}$, as displayed on page 483 of \cite{cransfioresatyendra}.}\label{table:consonant_triads}
\end{table}

\subsection{Transposition and Inversion, and $P$ and $L$}\label{subsec:TI-PL} The twelve-tone operations {\it transposition} $T_n\co \mathbb{Z}_{12} \to \mathbb{Z}_{12}$ and {\it inversion} $I_n\co \mathbb{Z}_{12} \to \mathbb{Z}_{12}$ are
\begin{displaymath}
T_n(x)=x+n \hspace{.75in} \text{ and } \hspace{.75in} I_n(x)=-x+n
\end{displaymath}
for $n \in \mathbb{Z}_{12}$.
These 24 operations are the symmetries of the 12-gon, when we consider $0$ through $11$ as arranged on the face of a clock. In the music-theory tradition, this group is called the $T/I$-group (the ``/'' does {\it not} indicate any kind of quotient). The unique reflection of the 12-gon which interchanges $m$ and $n$ is $I_{m+n}$, as can be verified by direct computation.

Many composers, for instance Schoenberg, Berg, and Webern, utilized these {\it mod} 12 transpositions and inversions. These functions and their compositional uses have been thoroughly explored by composers, music theorists, and mathematicians, see for example Babbitt \cite{Babbitt_SomeAspects}, Forte \cite{Forte_Structure}, Fripertinger--Lackner \cite{FripertingerLackner_ToneRowsAndTropes}, Hook \cite{HookCombinatorial}, Hook--Peck \cite{HookPeckIntroduction}, McCartin \cite{McCartinPrelude}, Mead \cite{MeadRemarks}, Morris \cite{MorrisCompositionWithPitchClasses,MorrisClassNotes,MorrisClassNotesAdvanced,MorrisRemarks}, and Rahn \cite{RahnBasicAtonalTheory}. Indeed, the three recent papers \cite{FripertingerLackner_ToneRowsAndTropes,MeadRemarks,MorrisRemarks} together contain over 100 references.

We consider these bijective functions on $\mathbb{Z}_{12}$ also as bijective functions $\text{\it Triads} \to \text{\it Triads}$ via their componentwise evaluation on consonant triads:
\begin{equation} \label{equ:TI-action_on_triads}
T_n\langle x_1,x_2,x_3 \rangle=\langle T_n x_1,T_n x_2, T_n x_3 \rangle\hspace{.25in} \text{ and } \hspace{.25in} I_n\langle x_1,x_2,x_3 \rangle=\langle I_n x_1,I_n x_2, I_n x_3 \rangle.
\end{equation}

Also on the set $\text{\it Triads}$ of consonant triads (with the indicated ordering), but not on the level of individual pitch classes, we have the bijective functions $P,L\co \text{\it Triads} \to \text{\it Triads}$ defined by
\begin{equation} \label{equ:PL-action_on_triads}
P\langle x_1,x_2,x_3 \rangle=I_{x_1+x_3}\langle x_1,x_2,x_3 \rangle\hspace{.25in} \text{ and } \hspace{.25in} L\langle x_1,x_2,x_3 \rangle=I_{x_2+x_3}\langle x_1,x_2,x_3 \rangle.
\end{equation}
As remarked above, $P$ stands for ``parallel'' and $L$ stands for ``leading tone exchange.''

We consider $T_n$, $I_n$, $P$, and $L$ as elements of the symmetric group $\text{\rm Sym}(\text{\it Triads})$.

\begin{proposition} \label{prop:PL_TI_commute}
The bijections $P$ and $L$ commute with $T_n$ and $I_n$ as elements of the symmetric group $\text{\rm Sym}(\text{\it Triads})$.
\end{proposition}
\begin{proof}
This is a straightforward computation using equations \eqref{equ:TI-action_on_triads} and \eqref{equ:PL-action_on_triads}. This computation has been discussed in broader contexts in \cite{fiorenollsatyendraSchoenberg} and \cite{fioresatyendra2005}.
\end{proof}

\subsection{Orbit-Stabilizer Theorem}

Suppose $S$ is a set with a left group action by a group $G$. Recall that the \emph{orbit of an element $Y\in S$} is
$${\text{\rm orbit of } Y}:=\{gY \mid g\in G \}.$$
The \emph{stabilizer group of an element $Y \in S$} is
$$G_Y:=\{g\in G \mid gY=Y \}.$$

\begin{theorem}[Orbit-Stabilizer Theorem]
\label{thm:Orbit Stabilizer Theorem}
Let $G$ be a group with an action on a set $S$. Neither $G$ nor $S$ is assumed to be finite. Then the assignment
$$\xymatrix{G/G_Y \ar[r] & \text{\rm orbit of } Y}$$
$$\xymatrix{gG_Y \ar@{|->}[r] & gY}$$
is a bijection. In particular, if $G$ is finite, then each orbit is finite, and
\begin{equation} \label{equ:OS}
\left\vert{G}\right\vert/\left\vert{G_Y}\right\vert=\left\vert{\text{\rm orbit of } Y}\right\vert.
\end{equation}
\end{theorem}

\subsection{Simple Transitivity} A group action of a group $G$ on a set $S$ is said to be {\it simply transitive} if for any $Y,Z \in S$ there is a unique $g \in G$ such that $gY=Z$. Informally, we also say the group $G$ is {\it simply transitive} if the sole action under consideration is simply transitive.

\begin{proposition} \label{prop:simple_transitivity_iff_each_stabilizer_trivial}
\leavevmode
\begin{enumerate}
\item \label{prop:simple_transitivity_iff_each_stabilizer_trivial:i}
An action of a group $G$ on a set $S$ is simply transitive if and only if it is transitive and every stabilizer $G_Y$ is trivial.
\item \label{prop:simple_transitivity_iff_each_stabilizer_trivial:ii}
Suppose $G$ is a finite group that acts on a set $S$. Then $G$ is simply transitive if and only if any two of the following three hold.
\begin{enumerate}
\item \label{a}
$G$ is transitive.
\item \label{b}
Every stabilizer $G_Y$ is trivial.
\item \label{c}
$G$ and $S$ have the same cardinality.
\end{enumerate}
In this case, the third condition also holds. \\ Another way to read this ``if and only if'' statement is: assuming $G$ is finite and any one of the conditions holds, $G$ is simply transitive if and only if another one of the conditions holds.
\item \label{prop:simple_transitivity_iff_each_stabilizer_trivial:iv}
Suppose a (not necessarily finite) group $H_1$ acts simply transitively on a set $S$, and a subgroup $H_2$ of $H_1$ acts  transitively on $S$ via its subaction. Then $H_1=H_2$.
\end{enumerate}
\end{proposition}
\begin{proof}
\begin{enumerate}
\item
If the action is simply transitive, then it acts transitively and for each $Y \in S$, there is only one $g \in G$ with $gY=Y$, and hence each $G_Y$ is trivial.

Suppose $G$ acts transitively and for every $Y \in S$, the group $G_Y$ is trivial. Suppose $Y, Z \in S$ and $g_1, g_2 \in G$ satisfy $g_1 Y =Z$ and $g_2 Y =Z$. Then $Y=g_2^{-1}Z$ and $g_2^{-1} g_1 Y=Y$, so $g_2^{-1} g_1 \in G_Y=\{e\}$, and finally $g_1 =g_2$.
\item
We first prove that any two of the conditions implies the third and implies simple transitivity. \\
\ref{a}\ref{b}$\Rightarrow$ $G$ is simply transitive by \ref{prop:simple_transitivity_iff_each_stabilizer_trivial:i}, and equation \eqref{equ:OS} says $|G|/1=|S|$, so $|G|=|S|$ and \ref{c} holds. \\
\ref{b}\ref{c}$\Rightarrow$ Equation \eqref{equ:OS} says $|S|=|G|/1=|\text{orbit of }Y|$, so $S=\text{orbit of $Y$}$, and $G$ is transitive and \ref{a} holds, so $G$ is simply transitive by \ref{prop:simple_transitivity_iff_each_stabilizer_trivial:i}. \\
\ref{a}\ref{c}$\Rightarrow$ Equation \eqref{equ:OS} says $|G|/|G_Y|=|G|$, so $|G_Y|=1$ and \ref{b} holds, and $G$ is simply transitive by \ref{prop:simple_transitivity_iff_each_stabilizer_trivial:i}. \\
Now that we have shown any two of the conditions implies the third and simple transitivity, we want to see that simply transitivity implies all three conditions. From \ref{prop:simple_transitivity_iff_each_stabilizer_trivial:i}, simple transitivity implies \ref{a} and \ref{b}, and we have already seen \ref{a} and \ref{b} imply \ref{c}.
\item
Suppose $H_1$ properly contains $H_2$, and $h_1 \in H_1 \backslash H_2$. Fix a $Y \in S$ and define $Z:=h_1Y$. Then by the transitivity of $H_2$, there is an $h_2 \in H_2$ such that $Z=h_2Y$. But by the simple transitivity of $H_1$, we must have $h_1=h_2$, a contradiction.
\end{enumerate}
\end{proof}

\section{Main Theorem: Hexatonic Duality} \label{sec:Main_Theorem}

We next review the notion of dual groups, and then turn to the main result, Theorem~\ref{thm:hexatonic_duality} on Hexatonic Duality.  Recall that subgroups $G$ and $H$ of $\text{Sym}(S)$ are {\it dual in the sense of Lewin} \cite[page 253]{LewinGMIT} if each acts simply transitively on $S$ and each is the centralizer of the other.\footnote{Lewin did not formally make this definition, but on page 253 of \cite{LewinGMIT} he gave a more general situation that gives rise to examples of dual groups in the sense defined above. He starts with a group $G$, there called $STRANS$, assumed to act simply transitively on a set $S$, and then makes three claims without proof: 1) the centralizer $C(G)$ in $\text{\rm Sym}(S)$ acts simply transitively on $S$ (the centralizer $C(G)$ is called $STRANS'$ there); 2) the double centralizer $C(C(G))$ is contained in $G$, so actually $C(C(G))=G$; and 3) the two generalized interval systems with transposition groups $G$ and $C(G)$ respectively have interval preserving transformation groups precisely $C(G)$ and $G$ respectively. See Proposition~\ref{prop:dual_groups_ala_Lewin} for a proof of statements 1) and 2). Statement 3) is a consequence of the first two statements in combination with $COMM$-$SIMP$ duality, which was stated by Lewin on page 101 of \cite{LewinCommSimp} and partially proved in \cite[Theorem~3.4.10]{LewinGMIT}. For a review of $COMM$-$SIMP$ duality and more proof, see Fiore--Satyendra \cite[Section 2 and Appendix]{fioresatyendra2005}. For the equivalence of generalized interval systems and simply transitive group actions, see pages 157--159 of Lewin's monograph. The equivalence on the level of categories was proved by Fiore--Noll--Satyendra on page 10 of \cite{fiorenollsatyendraSchoenberg}. The undergraduate research project \cite{sternberg} of Sternberg worked out some of the details of Lewin's simply transitive group action associated to a generalized interval system and investigated the Fugue in F from Hindemith's {\it Ludus Tonalis}.}  Recall the \emph{centralizer of $G$ in $\text{\rm Sym}(S)$} is $$C(G)= \{\sigma \in \text{\rm Sym}(S) \mid \sigma g= g \sigma \; \text{ for all } \; g \in G\}.$$

Before turning to the main result, we prove two simultaneous redundancies in the notion of {\it dual groups}: instead of requiring the two groups to centralize each other, it is sufficient to merely require that they commute, and instead of requiring $H$ to act simply transitively, it is sufficient to merely require $H$ acts transitively.

\begin{proposition} \label{prop:Commutativity_is_Enough_for_Dual_Groups}
Let $S$ be a (not necessarily finite) set. Suppose $G \leq \text{\rm Sym}(S)$ acts simply transitively on $S$ and $H \leq \text{\rm Sym}(S)$ acts transitively on $S$. Suppose $G$ and $H$ commute in the sense that $gh=hg$ for all $g \in G$ and $h \in H$. Then $G$ and $H$ are dual groups. In particular, $H$ also acts simply transitively and $G$ and $H$ centralize one-another.
\end{proposition}
\begin{proof}
We would like to first conclude from the simple transitivity of $G$, the transitivity of $H$, and the commutativity of $G$ and $H$, that the centralizer $C(G)$ acts simply transitively on $S$.

We claim that $C(G)$ acts simply transitively on $S$. It acts transitively, as $C(G) \supseteq H$ and $H$ acts transitively.
So, it suffices by Proposition~\ref{prop:simple_transitivity_iff_each_stabilizer_trivial}~\ref{prop:simple_transitivity_iff_each_stabilizer_trivial:i} to prove that, for each $s \in S$, the only element of $C(G)$ that fixes $s$ is the identity. Let $\sigma$ be an element of $C(G)$ that fixes $s$, and $g$ any element of $G$. Then,
$$
\aligned
\sigma s &= s \\
g \big( \sigma s \big) &= g \big( s \big) \\
\big( g \sigma \big) s  &= \big( g  s \big) \\
\big( \sigma g \big) s &= \big( g s \big) \\
\sigma \big( g s \big) &= \big( g s\big).
\endaligned
$$
So, not only does $\sigma$ fix $s$, but $\sigma$ also fixes $\big( g s\big)$ for every $g \in G$. That is to say $\sigma=\text{Id}_{S}$, and $C(G)$ acts simply transitively on $S$.

Now we have the transitive subgroup $H$ contained in the simply transitive group $C(G)$ by the assumed commutativity, so by Proposition~\ref{prop:simple_transitivity_iff_each_stabilizer_trivial}~\ref{prop:simple_transitivity_iff_each_stabilizer_trivial:iv},
$H=C(G)$, and $H$ also acts simply transitively.

To obtain $C(H) = G$, we use the newly achieved simple transitivity of $H$ and repeat the argument with the roles of $G$ and $H$ reversed.
\end{proof}

We may now use a result of Dixon--Mortimer to prove what Lewin stated on page 253 of \cite{LewinGMIT}, as suggested by Julian Hook, Robert Peck, and Thomas Noll. Parts \ref{prop:dual_groups_ala_Lewin:i} and \ref{prop:dual_groups_ala_Lewin:ii} of the following proposition were stated by Lewin.

\begin{proposition} \label{prop:dual_groups_ala_Lewin}
Let $S$ be a (not necessarily finite) set. Suppose $G \leq \text{\rm Sym}(S)$ acts simply transitively on $S$. Then
\begin{enumerate}
\item \label{prop:dual_groups_ala_Lewin:i}
The centralizer $C(G)$ in $\text{\rm Sym}(S)$ acts simply transitively on $S$.
\item \label{prop:dual_groups_ala_Lewin:ii}
The centralizer of the centralizer $C(C(G))$ is equal to $G$.
\item
Define $H:=C(G)$. Then $G$ and $H$ are dual groups.
\end{enumerate}
\end{proposition}
\begin{proof}
\begin{enumerate}
\item
This follows immediately from Dixon--Mortimer's \cite[Theorem 4.2A (i) and (ii), page 109]{dixonmortimer}. There {\it semi-regular} means point stabilizers are trivial and {\it regular} means simply transitive.
\item
Since $C(G)$ is simply transitive, we can apply Dixon--Mortimer's result to $C(G)$ to have the double centralizer $C(C(G))$ simply transitive. But $C(C(G))$ contains the simply transitive group $G$, so $C(C(G))=G$ by Proposition~\ref{prop:simple_transitivity_iff_each_stabilizer_trivial}~\ref{prop:simple_transitivity_iff_each_stabilizer_trivial:iv}.
\item
This follows directly from the preceding two by definition.
\end{enumerate}
\end{proof}

We now turn to the discussion of our main result.

Let $\text{\it Hex}$ be the set of chords in the hexatonic cycle \eqref{equ:hexatonic_cycle_Eflat} and $\underline{\text{\it Hex}}$ the set of underlying pitch classes of its chords, that is,
$$\text{\it Hex}:=\{E\flat,e\flat,B,b,G,g\},$$
$$\underline{\text{\it Hex}}:=\{2,3,6,7,10,11\}.$$

Our goal is to prove that the restriction of the $PL$-group to $\text{\it Hex}$ and the restriction of $\{T_0, T_4, T_8, I_1, I_5, I_9\}$ to $\text{\it Hex}$ are dual groups, and that each is dihedral of order 6. The strategy is to separately prove the unrestricted groups act simply transitively and are dihedral, and then finally to show that the restricted groups centralize each other. We begin with a characterization of the consonant triads contained in $\underline{\text{\it Hex}}$.

\begin{lemma} \label{lem:only_chords_in_Hex}
The only consonant triads of Table~\ref{table:consonant_triads} contained in $\underline{\text{\it Hex}}$ as subsets are the elements of $\text{\it Hex}$.
\end{lemma}
\begin{proof}
We first identify the available perfect fifths in $\underline{\text{\it Hex}}$ (pairs with difference 7), and then check if the corresponding major/minor thirds are in $\underline{\text{\it Hex}}$.

The only pairs of the form $\langle x, x+7 \rangle$ are $\langle 3, 10 \rangle$, $\langle 7, 2 \rangle$, and $\langle 11, 6 \rangle$, and we see that $x+4$ is contained in $\underline{\text{\it Hex}}$ in each case, that is, 7, 11, and 3 are in $\underline{\text{\it Hex}}$. Thus we have the three major chords $E\flat$, $G$, and $B$, and no others.

The only pairs of the form $\langle x+7, x \rangle=\langle y, y+5 \rangle$ are $\langle 2, 7 \rangle$, $\langle 6, 11 \rangle$, and $\langle 10, 3 \rangle$, and we see that $y+8$ is contained in $\underline{\text{\it Hex}}$ in each case, that is, 10, 2, and 6 are in $\underline{\text{\it Hex}}$. Thus we have the three minor chords $g$, $b$, and $e\flat$, and no others.
\end{proof}

\begin{proposition} \label{prop:H-dihedral_of_order_6}
\hspace{4in}
\begin{enumerate}
\item
The only elements of the $T/I$-group that preserve $\text{\it Hex}$ as a set are $\{T_0, T_4, T_8, I_1, I_5, I_9\}$, so they form a group denoted $H$.
\item
$H:=\{T_0, T_4, T_8, I_1, I_5, I_9\}$ is dihedral of order 6.
\end{enumerate}
\end{proposition}
\begin{proof}
\begin{enumerate}
\item
If an element of the $T/I$-group preserves $\text{\it Hex}$ as a set, than it must also preserve the collection $\underline{\text{\it Hex}}$ of underlying pitch classes as a set. Geometric inspection of the plot of $\underline{\text{\it Hex}}$ in Figure~\ref{fig:reflections} reveals that the only rotations that preserve $\underline{\text{\it Hex}}$ are $T_0$, $T_4$, and $T_8$. \begin{figure}[h]
  \centering
  \includegraphics[width=2.5in]{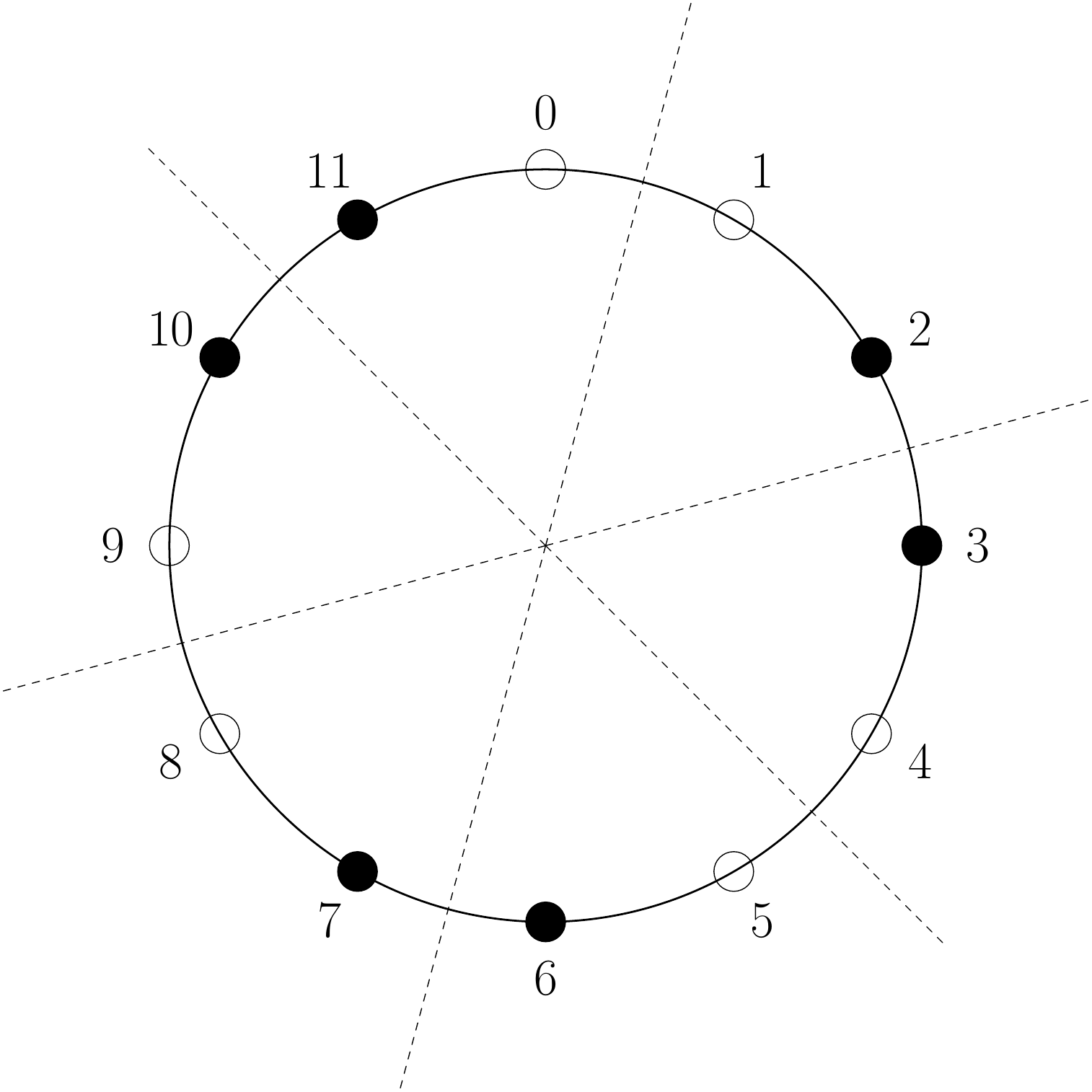}\\
  \caption{The solid circles represent the subset $\underline{\text{\it Hex}}$ of $\mathbb{Z}_{12}$. The symmetry of the subset makes apparent that the only rotations which preserve $\underline{\text{\it Hex}}$ are $T_0$, $T_4$, and $T_8$. The geometric locations of the solid circles also imply that the reflections across the dashed lines are the only reflections which preserve $\underline{\text{\it Hex}}$.}\label{fig:reflections}
\end{figure}

Again looking at Figure~\ref{fig:reflections}, we see that the three reflections which interchange $2\leftrightarrow 3$, or $6 \leftrightarrow  7$, or $10 \leftrightarrow 11$ preserve $\underline{\text{\it Hex}}$. By a comment in Section~\ref{subsec:TI-PL}, these are
$$I_{2+3}=I_5 \hspace{1in} I_{6+7}=I_1 \hspace{1in} I_{10+11}=I_9.$$
No other reflections preserve $\underline{\text{\it Hex}}$ as we can see geometrically from the limited reflection symmetry of $\underline{\text{\it Hex}}$.

Since $H:=\{T_0, T_4, T_8, I_1, I_5, I_9\}$ is a set-wise stabilizer of $\underline{\text{\it Hex}}$, it is a group.

From Lemma~\ref{lem:only_chords_in_Hex} we see that $\{T_0, T_4, T_8, I_1, I_5, I_9\}$ must also stabilize the chord collection $\text{\it Hex}$ as a set. No other transpositions or inversions stabilize $\text{\it Hex}$ by the argument at the outset of this proof.
\item
The only non-commutative group of order 6 is the symmetric group on three elements, denoted $\text{\rm Sym}(3)$, which is isomorphic to the dihedral group of order 6. The group under consideration is non-commutative, because $T_4 I_1(x)=-x+5$ while $I_1 T_4(x)=-x-3$.
\end{enumerate}
\end{proof}

\begin{proposition} \label{prop:H_simply_transitive}
The set-wise stabilizer $H$ acts simply transitively on $\text{\it Hex}$.
\end{proposition}
\begin{proof}
The $H$-orbit of $E\flat$ is all of $\text{\it Hex}$, as the following computation shows.
$$
\xymatrix@C=2.5pc@R=2.5pc{*=<1.7pc>[o][F-]{E\flat} &  & *=<1.7pc>[o][F-]{e\flat} \\ *=<1.7pc>[o][F-]{G} & *=<1.7pc>[o][F-]{E\flat} \ar[ul]_{T_0} \ar[l]_{T_4} \ar[dl]^{T_8} \ar[ur]^{I_1} \ar[r]^{I_5} \ar[dr]_{I_9} & *=<1.7pc>[o][F-]{g} \\ *=<1.7pc>[o][F-]{B} &  & *=<1.7pc>[o][F-]{b} }$$
We have $\vert H\vert =6 = \vert {\text{\rm orbit of } \ Y}\vert $
so the Orbit-Stabilizer Theorem $$\left\vert{H}\right\vert/\left\vert{H_Y}\right\vert=\left\vert{\text{\rm orbit of } \ Y}\right\vert$$
implies $\left\vert{H_Y}\right\vert=1$. See Proposition~\ref{prop:simple_transitivity_iff_each_stabilizer_trivial}~\ref{prop:simple_transitivity_iff_each_stabilizer_trivial:ii}.
\end{proof}

Next we can investigate the subgroup of $\text{\rm Sym}(\text{\it Triads})$ generated by $P$ and $L$, which is called the {\it $PL$-group}.

\begin{proposition} \label{prop:PL-dihedral_of_order_6}
The subgroup $\langle P, L\rangle$ of $\text{\rm Sym}(\text{\it Triads})$ is dihedral of order 6.
\end{proposition}
\begin{proof}
We first observe that $P$ and $L$ are involutions, that is, $P^2=\text{Id}_{\text{\it Triads}}$ and $L^2=\text{Id}_{\text{\it Triads}}$.  A musical justification is the definition of ``parallel'' and ``leading tone exchange.'' A mathematical justification is a direct computation of $P^2$ and $L^2$ using the formulas in \eqref{equ:PL-action_on_triads}.

Since $P$ and $L$ are involutions, every nontrivial element of $\langle P, L\rangle$ can be expressed as an alternating word in the letters $P$ and $L$. The six functions $\text{Id}_{\text{\it Triads}}$, $P$, $LP$, $PLP$, $LPLP$, and $PLPLP$ are all distinct by evaluating at $E\flat$ using the following diagram from the Introduction.
\begin{equation} \label{equ:Eflat-PL-cycle-for-proof}
\entrymodifiers={=<1.5pc>[o][F-]}
\xymatrix@C=3.5pc@R=3.5pc{E\flat \ar[r]^{P}  &  e\flat \ar[r]^L & B \ar[r]^P & b \ar[r]^L & G \ar[r]^P & g \ar@/^4pc/[lllll]^L }
\end{equation}
From diagram \eqref{equ:Eflat-PL-cycle-for-proof} we also see that $(LP)^3(E\flat)=E\flat$, and for any $Y\in\{E\flat,B,G\}$, $(LP)^3(Y)=Y$. Similarly, by reading the diagram backwards (recall $P$ and $L$ are involutions), we see $(LP)^3(Y)=Y$ for any minor triad $Y \in \{e\flat,b,g\}$.
We have similar $PL$-diagrams and considerations for the cycles in \eqref{equ:Hex2}, \eqref{equ:Hex3}, and \eqref{equ:Hex4}, and therefore $(LP)^3=\text{Id}_{\text{\it Triads}}$ on the entire set $\text{\it Triads}$ of consonant triads. Another way to see that $(LP)^3=\text{Id}_{\text{\it Triads}}$ is to combine the observation $(LP)^3(E\flat)=E\flat$ from diagram \eqref{equ:Eflat-PL-cycle-for-proof} with Proposition~\ref{prop:PL_TI_commute} and the fact that $\text{\it Triads}$ is the $T/I$-orbit of $E\flat$.

We next show via a word-theoretic argument that $\langle P, L\rangle$ consists only of the six functions $\text{Id}_{\text{\it Triads}}$, $P$, $LP$, $PLP$, $LPLP$, and $PLPLP$ discussed above. From $(LP)^3=\text{Id}_{\text{\it Triads}}$, we express $PL$ in terms of $LP$. Namely,
$$\aligned
(LP)^3 &= \text{Id}_{\text{\it Triads}} \\
(LP)^3(PL) &= (PL) \\
(LP)^2 &= PL.
\endaligned$$
Consider any alternating word in $P$ and $L$. If the right-most letter is $P$, then we can use $(LP)^3=\text{Id}_{\text{\it Triads}}$ to achieve an equality with one of the six functions we already have. If the right-most letter is $L$, then we replace each $PL$ by $(LP)^2$ and use $L^2=\text{Id}_{\text{\it Triads}}$ if $LL$ results on the far left. Then we have an equal function with right-most letter $P$, which we can then reduce to one of the six above using $(LP)^3=\text{Id}_{\text{\it Triads}}$ as we did in the first case of right-most letter $P$. Thus
$\langle P, L\rangle=\{\text{Id}_\text{\it Triads},\; P,\; LP,\; PLP,\; LPLP,\; PLPLP\}$.

This group is non-commutative, as $PL \neq LP$, hence it is isomorphic to $\text{Sym}(3)$, the only non-commutative group of order 6. But $\text{Sym}(3)$ is dihedral of order 6.

Instead of the previous paragraph, we can show $\langle P, L\rangle$ is dihedral of order 6 using a presentation.  Let $t:=L$ and $s:=LP$, then  $s^3=e$, $t^2=e$, and $tst=s^{-1}$. The dihedral group of order 6 is the largest group with elements $s$ and $t$ such that $s^3=e$, $t^2=e$, and $tst=s^{-1}$. But we observed from diagram \eqref{equ:Eflat-PL-cycle-for-proof} that $\langle P, L \rangle$ has at least six distinct elements. Hence, $\langle P, L \rangle$ is dihedral of order 6.
\end{proof}

\begin{proposition} \label{prop:PL_simply_transitive}
The $PL$-group $\langle P, L\rangle$ acts simply transitively on $\text{\it Hex}$.
\end{proposition}
\begin{proof}
From diagram \eqref{equ:Eflat-PL-cycle-for-proof} we see that $\langle P, L\rangle$ acts transitively on $\text{\it Hex}$. Since $\langle P, L\rangle$ and $\text{\it Hex}$ have the same cardinality, the Orbit-Stabilizer Theorem implies that every stabilizer must be trivial. See Proposition~\ref{prop:simple_transitivity_iff_each_stabilizer_trivial}~\ref{prop:simple_transitivity_iff_each_stabilizer_trivial:ii}.
\end{proof}

\begin{lemma} \label{lem:restriction_gives_iso_and_simple_transitivity}
Let $S$ be a set and suppose $G \leq \text{\rm Sym}(S)$. Suppose $G$ acts simply transitively on an orbit $\overline{S}$, and $\overline{G}$ is the restriction of $G$ to the orbit $\overline{S}$. Then the restriction homomorphism $G \to \overline{G}$ is an isomorphism, and $\overline{G}$ also acts simply transitively.
\end{lemma}
\begin{proof}
Suppose $g\in G$ has restriction $\overline{g}$ with $\overline{g} \,\overline{s}=\overline{s}$ for all $\overline{s} \in \overline{S}$. Then unrestricted $g$ also has $g\overline{s}=\overline{s}$ for all $\overline{s} \in \overline{S}$, so $g=\text{Id}_S$ by simple transitivity, and the kernel of the surjective homomorphism $G \to \overline{G}$ is trivial. The transitivity of $\overline{G}$ is clear: for any $\overline{s},\overline{t} \in \overline{S}$ there exists $g \in G$ such that $g\overline{s}=\overline{t}$, so also $\overline{g}\hspace{.4mm} \overline{s}=\overline{t}$ with $\overline{g}\in \overline{G}$. The uniqueness of $\overline{g} \in \overline{G}$ is also clear: if $\overline{h} \in \overline{G}$ also satisfies $\overline{h}\hspace{.4mm} \overline{s}=\overline{t}$, then so do the unrestricted $g$ and $h$, so $g=h$ by the simple transitivity of $G$ acting on $\overline{S}$, so $\overline{g}=\overline{h}$.
\end{proof}

\begin{theorem}[Hexatonic Duality] \label{thm:hexatonic_duality}
The restrictions of the $PL$-group and $\{T_0, T_4, T_8, I_1, I_5, I_9\}$ to $\text{\it Hex}$ are dual groups in $\text{\rm Sym}(\text{\it Hex})$, and both are dihedral of order 6.
\end{theorem}
\begin{proof}
Let $\overline{G}$ be the restriction of the $PL$-group to $\text{\it Hex}$, and let $\overline{H}$ be the restriction of
$H=\{T_0, T_4, T_8, I_1, I_5, I_9\}$ to $\text{\it Hex}$.

We already know that $\overline{G}$ and $\overline{H}$ are dihedral of order 6 by Propositions~\ref{prop:H-dihedral_of_order_6} and \ref{prop:PL-dihedral_of_order_6} and Lemma~\ref{lem:restriction_gives_iso_and_simple_transitivity}.

We also already know that $\overline{G}$ and $\overline{H}$ each act simply transitively on $\text{\it Hex}$ by Propositions~\ref{prop:H_simply_transitive} and \ref{prop:PL-dihedral_of_order_6} and Lemma~\ref{lem:restriction_gives_iso_and_simple_transitivity}. We even already know that the groups $\overline{G}$ and $\overline{H}$ commute by Proposition~\ref{prop:PL_TI_commute}. Finally, Proposition~\ref{prop:Commutativity_is_Enough_for_Dual_Groups} guarantees that $\overline{G}$ and $\overline{H}$ centralize one another.
\end{proof}

\begin{rmk}[Comparison with the Proof Strategy of Crans--Fiore--Satyendra in \cite{cransfioresatyendra}] \label{rem:differences_and_similarities}
There are several differences between the proof strategy of Hexatonic Duality in present Theorem~\ref{thm:hexatonic_duality} and the proof strategy of $T/I$-$PLR$ duality in Theorem~6.1 of \cite{cransfioresatyendra}. In the present paper, we first proved that the concerned groups act simply transitively, and determined their structure, and only then showed that the groups exactly centralize each other. In \cite{cransfioresatyendra}, on the other hand, the determination of the size of the $PLR$-group was postponed until after the centralizer $C(T/I)$ was seen to act simply, i.e. that each stabilizer $C(T/I)_Y$ is trivial. Then, from these trivial stabilizers, the Orbit-Stabilizer Theorem, the earlier observation that $24 \leq |PLR\text{-group}|$, and the consequence
$$24 \leq |PLR\text{-group}| \leq |C(T/I)| \leq |\text{orbit of $Y$}| \leq 24$$
on page 492, the authors of \cite{cransfioresatyendra} simultaneously conclude that the $PLR$-group has 24 elements and is the centralizer of $T/I$.

A slight simplification of the aforementioned inequality would be an argument like the one in the present paper: observe that the $PLR$-group acts transitively on the 24 consonant triads because of the Cohn $LR$-sequence (recalled on page 487 of \cite{cransfioresatyendra}), then $C(T/I)$ must act transitively as it contains the $PLR$-group, and then the Orbit-Stabilizer Theorem and the trivial stabilizers imply that $|C(T/I)|$ must be 24, so the $PLR$-group also has 24 elements. Also, instead of postponing the proof that the $PLR$-group has exactly 24 elements from Theorem~5.1 of \cite{cransfioresatyendra} until the aforementioned inequality in Theorem~6.1, one could do a word-theoretic argument in Theorem~5.1 to see that the $PLR$-group has exactly 24 elements, similar to the present argument in Proposition~\ref{prop:PL-dihedral_of_order_6}.
\end{rmk}

\begin{rmk}
For an explicit computation of the four hexatonic cycles as orbits of the $PL$-group, see Oshita \cite{oshita}, which was also an undergraduate research project with the second author of the present article. The preprint \cite{oshita} includes a sketch that $\langle P,L \rangle\cong \text{Sym}(3)$.
\end{rmk}

\begin{rmk}[Alternative Derivation using the Sub Dual Group Theorem]
Hexatonic Duality Theorem~\ref{thm:hexatonic_duality} can also be proved using the Sub Dual Group Theorem~3.1 of Fiore--Noll in \cite{fiorenoll2011}, {\it if one assumes already} the duality of the $T/I$-group and $PLR$-group (maximal smoothness is not discussed in \cite{fiorenoll2011}). Fiore--Noll apply the Sub Dual Group Theorem to the construction of dual groups on the hexatonic cycles in \cite[Section~3.1]{fiorenoll2011}. The method is to select $G_0$ to be the $PL$-group, select $s_0=E\flat$, and compute $S_0:=G_0s_0=\text{\it Hex}$, and then the dual group will consist of the restriction of those elements of the $T/I$-group that map $E\flat$ into $S_0$.

Notice that in the present paper, on the other hand, we {\it first} determined which transpositions and inversions preserve $\text{\it Hex}$ in Proposition~\ref{prop:H-dihedral_of_order_6}, and then proved duality, whereas the application of the Sub Dual Group Theorem of Fiore--Noll starts with the $PL$-group and determines from it the dual group as (the restrictions of) those elements of the $T/I$-group that map $E\flat$ into $S_0$. Notice also, in the present paper we determined that the $PL$-group and its dual $H$ are dihedral of order 6, but that the Sub Dual Group Theorem of Fiore--Noll does not specify which group structure is present. In any case, Clampitt explicitly wrote down all 6 elements of each group in permutation cycle notation in \cite{clampittParsifal}.

The present paper is complementary to the work \cite{fiorenoll2011} of Fiore--Noll in that we work very closely with the specific details of the groups and sets involved to determine one pair of dual groups in an illustrative way, rather than appealing to a computationally and conceptually convenient theorem. Fiore--Noll however also use their Corollary~3.3 to compute the other hexatonic duals via conjugation, as summarized in Table~\ref{table:SubDualGroupTheorem_Application}.

\begin{table}[h]
  \centering
  $\begin{tabular}{|c|c|c|}
\hline $k$ & $k\text{\it Hex}$ & $kHk^{-1}=$ dual group to $PL$-group on $k\text{\it Hex}$\\ \hline
$\text{Id}_{\text{\it Triads}}$ & $\{E\flat,e\flat,B,b,G,g\}$ & $H=\{T_0, T_4, T_8, I_1, I_5, I_9\}$ \\ \hline
$T_1$ & $\{E,e,C,c,A\flat,a\flat\}$ & $\{T_0, T_4, T_8,I_3, I_7, I_{11}\}$ \\ \hline
$T_2$ & $\{F,f,C\sharp, c\sharp, A,a\}$ & $\{T_0, T_4, T_8, I_5, I_9, I_1 \}$ \\ \hline
$T_3$ & $\{F\sharp,f\sharp,D,d,B\flat,b\flat\}$ & $\{T_0, T_4, T_8, I_7, I_{11}, I_3\}$ \\ \hline
\end{tabular}$
  \caption{The 4 hexatonic cycles as $PL$-orbits and the respective dual groups determined as conjugations of $H$ via the Sub Dual Group Theorem of Fiore--Noll.}\label{table:SubDualGroupTheorem_Application}
\end{table}

The application of the Sub Dual Group Theorem to construct dual groups on octatonic systems is also treated in \cite{fiorenoll2011}, and utilized in \cite{fiorenollsatyendraSchoenberg}.
\end{rmk}

\begin{rmk}[Other Sources on Group Actions]
Music-theoretical group actions on chords have been considered by many, many authors over the past century. In addition to the selected references of Babbitt, Forte, and Morris above, we also mention the expansive and influential work of Mazzola and numerous collaborators \cite{MazzolaGruppenUndKategorien,MazzolaGeometrieDerToene,MazzolaToposOfMusic}. Moreover, Issue 42/2 of the {\it Journal of Music Theory} from 1998 is illuminating obligatory reading on groups in neo-Riemannian theory. That issue contains Clampitt's article \cite{clampittParsifal}, which is the inspiration for the present paper. Clough's article \cite{cloughRudimentaryGeometricModel} in that issue illustrates the dihedral group of order 6 and its recombinations with certain centralizer elements in terms of two concentric equilateral triangles (Clough's article does not treat hexatonic systems and duality). The dihedral group of order 6 is a warm-up for his treatment of recombinations of the {\it Schritt-Wechsel} group with the $T/I$-group, which are both dihedral of order 24. Peck's article \cite{peckGeneralizedCommuting} studies centralizers where the requirement of simple transitivity is relaxed in various ways, covering many examples from music theory. Peck determines the structure of centralizers in several cases.
\end{rmk}

\begin{rmk}[Discussion of Local Diatonic Containment of Hexatonic Cycles] \label{rmk:diatonic_containment}
No hexatonic cycle is contained entirely in a single diatonic set, as one can see from any of the cycles \eqref{equ:Hex1}--\eqref{equ:Hex4}. However, one can consider a sequence of diatonic sets that changes along with the
hexatonic cycle and contains each respective triad, as Douthett does in \cite[Table 4.7]{douthettFIPSandDynVoiceLeading}. After transposing and reversing Douthett's table, we see a sequence of diatonic sets such that each diatonic set contains the respective triad of \eqref{equ:Hex1}.
$$\begin{tabular}{|l|c|c|c|c|c|c|}
\hline
Triad & $E\flat$ & $e\flat$ & $B$ & $b$ & $G$ & $g$ \\ \hline
In Scale & $E\flat\text{-major}$ & $D\flat\text{-major}$ & $B\text{-major}$ & $A\text{-major}$ & $G\text{-major}$ & $F\text{-major}$ \\ \hline
\end{tabular}$$
This sequence of diatonic sets (indicated via major scales) descends by a whole step each time, so is as evenly distributed as possible.

Other diatonic set sequences also contain the hexatonic cycle, though unfortunately there is no maximally smooth cycle of diatonic sets that does the job (recall that the diatonic sets can only form a cycle of length 12). But it is possible to have a maximally smooth sequence of diatonic sets that covers four hexatonic triads. We list all possible
diatonic sets containing the respective hexatonic chords.\footnote{Recall that major chords only occur with roots on major scale degrees 1, 4, and 5, so we determine in the table the scales containing a given major triad by considering the root, a perfect fourth below the root, and a perfect fifth below the root. Minor scales can only occur with roots on major scale degrees 2, 3, and 6, so we determine in the table the scales containing a given minor triad by considering a major sixth below the root, a whole step below the root, and a major third below the root. This non-consistent major/minor ordering allows us to see (at double vertical dividing lines) all three maximally smooth transitions from diatonic sets containing a given a minor triad to a diatonic set containing its subsequent major in a hexatonic cycle. }
$$\begin{tabular}{|l|c|c||c|c||c|c||}
\hline
Triad           & $E\flat$ & $e\flat$ & $B$       & $b$ & $G$ & $g$ \\ \hline
                & $E\flat$ & $G\flat$ & $B$       & $D$ & $G$ & $B\flat$ \\
In Major Scales & $B\flat$ & $\mathbf{D\flat}$ & $F\sharp$ & $\mathbf{A}$ & $D$ & $\mathbf{F}$ \\
                & $\mathbf{A\flat}$ & $B$      & $\mathbf{E}$       & $G$ & $\mathbf{C}$ & $E\flat$ \\ \hline
\end{tabular}$$
Double vertical dividing lines indicate maximally smooth transitions between consecutive diatonic sets. As indicated by double vertical dividing lines, the transition from a minor triad to its subsequent major in a hexatonic cycle via $L$ is contained in three maximally smooth transitions of diatonic sets. On the other hand, the transition from a major triad to its subsequent minor in a hexatonic cycle via $P$ is contained in only one maximally smooth transition of diatonic sets, as indicated by the bold letters. Altogether, we can trace three maximally smooth chains of four major scales that contain part of the hexatonic cycle \eqref{equ:Hex1}. $$B-E-A-D$$ $$G-C-F-B\flat$$ $$E\flat-A\flat-D\flat-F\sharp$$

Local containment of hexatonic cycles in diatonic chains has ramifications for music analysis. Jason Yust proposes in \cite{YustOndine} and \cite{YustDistortion} to include diatonic contexts into analyses involving $PL$-cycles or $PR$-cycles, and he provides analytical tools to do so.
\end{rmk}

\section{Conclusion} \label{sec:Conclusion}

We began this article with Cohn's proposal that the maximal smoothness of consonant triads is a key factor for their privileged status in late-nineteenth century music. Indeed, consonant triads and their complements are the only tone collections that accommodate short maximally smooth cycles. The four maximally smooth cycles of consonant triads, the so-called hexatonic cycles of Cohn, can be described transformationally as alternating applications of the neo-Riemannian ``parallel'' and ``leading tone exchange'' transformations. Cohn interpreted Wagner's Grail motive in terms of a cyclic group action on the hexatonic cycle $\text{\it Hex}$, whereas Clampitt used the $PL$-group and the transposition-inversion subgroup we called $H$ in Proposition~\ref{prop:H-dihedral_of_order_6}. In the present article, we proved the Lewinian duality between these latter two groups, which was discussed by Clampitt in \cite{clampittParsifal}.

For perspective, we mention that simply transitive group actions correspond to the {\it generalized interval systems} of Lewin, see the very influential original source \cite{LewinGMIT}, or see \cite[Section~2]{fiorenollsatyendraSchoenberg} for an explanation of some aspects. Dual groups correspond to dual generalized interval systems: the transpositions of one system are the interval preserving bijections of the other. Clampitt \cite{clampittParsifal} explained the coherent perceptual basis of the three generalized interval systems associated to the three group actions on $\text{\it Hex}$ by Cohn's cyclic group, the $PL$-group, and the $H$ group. He employed the coherence of generalized interval systems to incorporate the final $D\flat$ of the Grail motive into his interpretation via a conjugation-modulation of a subsystem.

\section*{Acknowledgements}
This paper is the extension of an undergraduate research project of student Cameron Berry with Professor Thomas Fiore at the University of Michigan-Dearborn. Cameron Berry thanks Professor Thomas Fiore for all of the time he spent assisting and encouraging him during Winter semester 2014. We both thank Mahesh Agarwal for a suggestion to use a word-theoretic argument in Theorem~5.1 of \cite{cransfioresatyendra}, which we implemented in the present Proposition~\ref{prop:PL-dihedral_of_order_6}. We thank Thomas Noll for proposing and discussing the extended Remark~\ref{rmk:diatonic_containment} with us. Thomas Fiore thanks Robert Peck, Julian Hook, David Clampitt, and Thomas Noll for a brief email correspondence that lead to Proposition~\ref{prop:simple_transitivity_iff_each_stabilizer_trivial}~\ref{prop:simple_transitivity_iff_each_stabilizer_trivial:iv},
Proposition~\ref{prop:simple_transitivity_iff_each_stabilizer_trivial}~\ref{prop:simple_transitivity_iff_each_stabilizer_trivial:ii},
and Proposition~\ref{prop:dual_groups_ala_Lewin}, and positively impacted Proposition~\ref{prop:Commutativity_is_Enough_for_Dual_Groups}. Thomas Fiore thanks Ramon Satyendra for introducing him to hexatonic cycles and generalized interval systems back in 2004.

Both authors thank the two anonymous referees for their constructive suggestions. These lead to the improvement of the manuscript.

Thomas Fiore was supported by a Rackham Faculty Research Grant of the University of Michigan. He also thanks the Humboldt Foundation for support during his 2015-2016 sabbatical at the Universit{\"a}t Regensburg, which was sponsored by a Humboldt Research Fellowship for Experienced Researchers. Significant progress on this project was completed during the fellowship.


\end{document}